\documentclass[11.05pt,a4paper]{amsart}
\usepackage{rotating}
\usepackage[all]{xy}

\usepackage[utf8]{inputenc}
\usepackage{amscd,amssymb,amsopn,amsmath,amsthm,graphics,amsfonts,enumerate,verbatim,calc}
\usepackage{mathtools}

\usepackage{hyperref}
\usepackage[capitalise]{cleveref}

\usepackage{setspace}
\usepackage{mathpazo}
\usepackage{color}
\usepackage{latexsym}
\usepackage{amsthm,amsfonts,amssymb,mathrsfs}
\usepackage{rotating}
\usepackage[leqno]{amsmath}
\usepackage{xspace}
\usepackage[all]{xy}
\usepackage{longtable}
\headsep=5 mm \numberwithin{equation}{section}
\newtheorem{theorem}{Theorem}[section]
\newtheorem{lemma}[theorem]{Lemma}

\newtheorem{proposition}[theorem]{Proposition}
\newtheorem{corollary}[theorem]{Corollary}

\newtheorem{problem}[theorem]{Problem}
\theoremstyle{definition}

\theoremstyle{remark}
\newtheorem{remark}[theorem]{Remark}
\newtheorem{fact}[theorem]{Fact}
\newtheorem{example}[theorem]{Example}

\newtheorem{observation}[theorem]{Observation}
\newtheorem{discussion}[theorem]{Discussion}
\newtheorem{question}[theorem]{Question}

\newcommand{\grade}{\operatorname{grade}}
\newcommand{\Tr}{\operatorname{Tr}}
\newcommand{\cd}{\operatorname{cd}}

\newcommand{\AB}{\operatorname{AB}}
\newcommand{\Spec}{\operatorname{Spec}}

\newcommand{\Ht}{\operatorname{ht}}
\newcommand{\pd}{\operatorname{pd}}

\newcommand{\iext}{\operatorname{Ext-index}}
\newcommand{\itor}{\operatorname{Tor-index}}

\newcommand{\Syz}{\operatorname{Syz}}

\newcommand{\id}{\operatorname{id}}
\newcommand{\Ext}{\operatorname{Ext}}
\newcommand{\up}[1]{{{}^{#1}\!}}

\newcommand{\Tor}{\operatorname{Tor}}

\newcommand{\Hom}{\operatorname{Hom}}

\newcommand{\Ann}{\operatorname{Ann}}

\newcommand{\depth}{\operatorname{depth}}

\newcommand{\vpl}{\operatornamewithlimits{\varprojlim}}

\newcommand{\vil}{\operatornamewithlimits{\varinjlim}}

\newcommand{\lo}{\longrightarrow}
\newcommand{\fm}{\frak{m}}
\newcommand{\fp}{\frak{p}}

\newcommand{\fn}{\frak{n}}

\begin{document}

\author[]{Mohsen Asgharzadeh}

\address{}
\email{mohsenasgharzadeh@gmail.com}

\title[ ]
{A note on cohomological vanishing theorems}


\begin{abstract} 
We study $ \cd(M,N):=\sup\{j:H^j_{\fm}(M,N)\neq0\}$, and we prove the following
over $\AB$-rings: 
  $\cd(M,N)<\infty \Rightarrow \cd(M,N)\leq2\dim R$. For   locally free over the punctured spectrum, we present the better  bound, namely $\cd(M,N)<\infty \Rightarrow \cd(M,N)\leq\dim R,$ 
and show this is sharp  for maximal Cohen-Macaulay, and prove that this detects freeness of $M$.
We present some explicit examples to compute $\cd(M,N)$. Now, suppose $R$ is only   Cohen-Macaulay and of prime characteristic equipped with the Frobenius map $\varphi$. We show for some $n\gg 0$ that  $\cd(\up{\varphi_n}R,M)<\infty$ iff   $\id_R(M)<\infty.$ This presents some criteria on regularity. Also, some vanishing results on $\Ext^i_R(\up{\varphi}R,-)$ are given,
where $(-)\in\{R,\up{\varphi}R\}$. We determine conditions under which the vanishing $\Ext^i_R(\up{\varphi}R,-)$ of restricted many $i$-th,  implies the vanishing of all.
\end{abstract}

\maketitle

\section{Introduction}

The Grothendieck's vanishing theorem says that
$H^i_{\fm}(M)=0$ for all $j> d:=\dim R$ where $(R,\fm,k)$
is noetherian and local. As soon as the ring is not regular
we see $H^\ast_{\fm}(k,k)\neq0$. Also, it   is well-known 
that $H^j_{\fm}(M,N)=0$ for all $j> d$ provided $\pd(M)<\infty$
or $\id(N)<\infty$ or even some modifications of these, see e.g. 
 \cite{HZ}.
Despite a lot of works in this area, and up to my knowledge, there is not so data on the following basic problem:
\begin{problem}
 What is $\cd(M,N):=\sup\{j:H^j_{\fm}(M,N)\neq0\}$?
\end{problem} 

And one may ask even more; what is the significance of $\cd(M,N)$? 
Towards the symmetry in the vanishing of $\Ext^i_R(-,+)$, for finitely generated modules over local Gorenstein rings, Huneke and Jorgensen \cite{AB} introduced 
``$\AB$-rings". This class of rings lives between complete-intersection and Gorenstein rings. We refer the reader to see \cite{AB} for the systematic study of $\AB$-rings.
Now, let $R$ be an $\AB$ ring. In Section 2 we prove the following:

\begin{enumerate}[(a)]
	\item  $\cd(M,N)<\infty \Rightarrow \cd(M,N)\leq2\dim R$.
	 
	\item  For $M$ be  locally free over the punctured spectrum, we have $$\cd(M,N)<\infty \Rightarrow  \cd(M,N)\leq\dim R.$$
	\item   For $M$ be maximal Cohen-Macaulay and  locally free over the punctured spectrum, we have $$\cd(M,M)<\infty\Longleftrightarrow \emph{M  is free},$$when $\dim R>1$.  In particular, $\cd(M,M)=\dim(R)$.
	\item For $M,N$ be maximal Cohen-Macaulay and  locally free over the punctured spectrum, we have  $\cd(M,N)=\cd(N,M)$.
   \end{enumerate}
The above item  (c) may be considered as the higher-dimensional version of recent work 
\cite[Corollary 4.12]{HM} with a connection to the famous  conjecture of Huneke and Wiegand. From now on $\varphi:R\to R$ denotes the \emph{Frobenius endomorphism}. In the sequel, we will see how that the above computations over $\AB$-rings motive us to prove the following result over the much larger class of rings: 
\begin{enumerate}
\item[(e)]	Let $R$ be a Cohen-Macaulay, F-finite and $n\gg 0$. Then $$\cd(\up{\varphi_n}R,M)<\infty\Longleftrightarrow  \id_R(M)<\infty.$$ 
\end{enumerate}Recall that $\id_R(-)$ denotes the injective dimension of $(-)$. Now, let $m>0$ be any integer and $R$ be   Cohen-Macaulay. Part (e) implies that $\cd(\up{\varphi_n}R,\up{\varphi_m}R)<\infty$ iff $R$  is regular. 
We also present the dual parallel to $(e)$, where  we do not assume  F-finiteness, see Proposition \ref{df}.
Here, we do not care about $n\gg 0$, provided the ring is complete-intersection, and even in the particular case $R$ is normal and locally $\AB$. 

To prove these results, we reduce the local cohomology problems to the vanishing of certain ext-modules. In Section 3 we continue with the later: \begin{problem} Let $-\in\{\up{\varphi_n}R,R\}$ and $i>0$.
 	When is $\Ext^i_R(\up{\varphi_n}R,-)=0$?
 \end{problem} The history of this comes back to \cite{hhe}, where  Herzog assumed and studied vanishing of all Ext-modules, see \cite{hhe}. It seems this is more subtle than the classical dual results   \cite{PS,Ku}. So, to the end of introduction, we assume the rings are F-finite.
In particular, we deal with the higher (lower) analogue of the recent work \cite[1.4]{two}. This naturally leads us to study the  reflexivity problem of $\up{\varphi_n}R$.  We determine conditions under which the vanishing of few many $i$, or  even a single value of $i$, $\Ext^i_R(\up{\varphi_n}R,-)$ implies the vanishing of all. Here, we summarized some of them:

\begin{enumerate}[(i)]
	\item  Suppose $\depth(R)=0$. Suppose
	$\Ext^1_R(\up{\varphi_n}R,\up{\varphi_n}R)=0$ for $n\gg 0$, then $R$ is a field.
	
	\item  	Suppose $R$ is complete intersection and
$\Ext^1_R(\up{\varphi_n}R,\up{\varphi_n}R)=0$  for some $n> 0$, then $R$ is regular.
	\item   Suppose $R$ is normal Cohen-Macaulay. Then $R$ is Gorenstein iff  $\up{\varphi_n}R$ is self-dual and $\Ext^{i}_R(\up{\varphi_n}R,R) = 0$ for $0<i\in\{d-1, d\}$.
	\item[(iv)]  We study the questions asked in \cite[Page 167]{mar} and  \cite[2.12]{l}, we show they are equivalent and present some results, both in positive and negative sides.
\end{enumerate}

Let us close the introduction, by recalling the relation between $\Ext^i_R(\up{\varphi}R,\omega_R)$ and the usual local cohomology by means of Grothendieck's local duality.

\section{When is $H_m^+(M,N)=0$? }
In this note $(R,\fm,k)$
is a commutative  noetherian and local ring, and all modules are finitely generated, otherwise socialized. 
For two $R$-modules $M$ and
$N$, the $i^{th}$ generalized local cohomology of $M$ and $N$  is defined by
$H^{i}_{\fm}(M,N):=\underset{n}{\varinjlim}\Ext^{i}_{R} (M/\fm^{n}M,
N)$.
\begin{proposition}\label{c}
Let $R$ be an $\AB$ ring, let $M$ and $N$ be finitely generated $R$-modules. If $\cd(M,N)<\infty$ then  	$\cd(M,N)\leq2\dim R$.
	\end{proposition}

\begin{proof}We may assume the ring is complete. We look at the short exact sequence $$0\lo N_1:=\Syz_1(N)\lo F\lo N\lo 0$$ where $F$ is free. This gives the exact sequence $$H^{i-1}_{\fm}(M,N)\lo H^j_{\fm}(M,N_1)\lo H^j_{\fm}(M,F)\quad(+)$$Recall that $\AB$-rings are Gorenstein, so the injective dimension of $F$ is finite. From $(+)$ we see $\cd(M,N_1)<\infty$.
Let $d:=\dim(R)$ and set $N_d:=\Syz_1(N)$. By repetition of this, we see $\cd(M,N_d)<\infty$. Suppose
the desired claim  holds with the additional  Cohen-Macaulay assumption of $N$. Since
$N_d$ is maximal Cohen-Macaulay we obtain that 
$\cd(M,N_d)\leq2d$. Let $j>2d$, and look at $$0=H^{j}_{\fm}(M,F)\lo H^j_{\fm}(M,N_{d-1})\lo H^{j+1}_{\fm}(M,N_d),$$we obtain that 
$H^j_{\fm}(M,N_{d-1})=0$. By repetition of this, we see $\cd(M,N)\leq2 d$. So, without loss of generality, we may and do assume that $N$ is Cohen-Macaulay. This allows us to apply \cite[Lemma 3.4]{HZ}
and to deduce for any $i \geq 0$ that  $$H^{d+i}_{\fm}(M,N)^{ \vee}\cong \Tor^R_i(M,H^{d}_{\fm}(R,N)^{ \vee})\cong \Tor^R_i(M,N^\ast).$$Here $(-)^{ \vee}$ is the Matlis functor. So, $\Tor^R_i(M,N^\ast)=0$ for all $i\gg 0$. But, tor-index of $\AB$-rings is   finite, and in fact  is equal to dimension of $R$ (see discussion after than \cite[Proposition 3.2]{AB}). Let $$\Sigma:=\{(L_1,L_2):\Tor^R_i(L_1,L_2)=0\emph{  for all } i\gg 0  \}.$$ By definition, tor-index is $$\sup\{n:\Tor^R_n(L_1,L_2)\neq 0\emph{ and }\Tor^R_{>n}(L_1,L_2)=0\}_{(L_1,L_2)\in\Sigma}.$$Let $j:=\max\{d,\cd(M,N)\}$. There is nothing to prove 
if $\cd(M,N)<d$, so we may assume that $i:=\cd(M,N)-d\geq 0$. In sum, $(M,N^\ast)\in\Sigma$ and that $$\Tor^R_i(M,N^\ast)\neq 0\emph{ and }\Tor^R_{>i}(M,N^\ast)=0.$$By definition, $$\cd(M,N)-\dim R=i\leq \itor(R)=\dim R,$$and so $\cd(M,N)\leq2\dim R$.
	\end{proof}

\begin{proposition}\label{d}
	Let $R$ be an $\AB$ ring, let $M$ and $N$ be finitely generated $R$-modules so that $M$ is locally free over the punctured spectrum. If $\cd(M,N)<\infty$ then  	$\cd(M,N)\leq\dim R$.
\end{proposition}

\begin{proof}We may assume the ring is complete. Without loss of generality, we may and do assume that $N$ is Cohen-Macaulay. Suppose $\cd(M,N)\geq d:=\dim (R)$, and let $i:=\cd(M,N)-d\geq 0$. We show that $i=0$. Indeed, 
recall that $$H^{d+i}_{\fm}(M,N)^{ \vee} \cong \Tor^R_i(M,N^\ast)\cong\Ext_R^{d+i}(M,N)^{ \vee}\quad(+),$$where the last isomorphism is taken from \cite[Lemma 3.5(2)]{ta}. But, ext-index of $\AB$-rings is  equal to  $\dim R$ (see \cite[Proposition 3.2]{AB}). Let $$\Sigma:=\{(L_1,L_2):\Ext_R^i(L_1,L_2)=0\emph{  for all } i\gg 0  \}.$$ By definition, ext-index is $$\sup\{n:\Ext_R^n(L_1,L_2)\neq 0\emph{ and }\Ext_R^{>n}(L_1,L_2)=0\}_{(L_1,L_2)\in\Sigma}.$$ In sum, $(M,N^\ast)\in\Sigma$ and that $$\Ext_R^{d+i}(M,N^\ast)\neq 0\emph{ and }\Ext_R^{>d+i}(M,N^\ast)=0.$$By definition, $$d\leq d+i\leq\iext(R)=\dim R,$$and so $i=0$ and consequently $\cd(M,N)\leq\dim R$.
\end{proof}

\begin{corollary}
	Let $R$ be an $\AB$ ring,  let $M$ and $N$ be maximal Cohen-Macaulay and locally free over the punctured spectrum. Then $\cd(M,N)<\infty \Longleftrightarrow \cd(N,M)<\infty$.
\end{corollary}
\begin{proof}
Let $d:=\dim(R)$.	In view of $(+)$ from Proposition \ref{d} it is  enough to note that  $$\Ext_R^{>d}(M,N)=0\Longleftrightarrow \Ext_R^{>d}(N,M)=0.$$This is subject of \cite[4.3]{AB}.
\end{proof}

\begin{example}
	The maximal Cohen-Macaulay of both modules are needed. For example, suppose $R$ is not regular, then $\cd(R,k)=0<\infty$ but $\cd(k,R)=\infty$.
\end{example}

\begin{corollary}\label{a}
	Let $R$ be an $\AB$ ring of dimension bigger than one, and $M$ be locally free over the punctured spectrum and maximal Cohen-Macaulay. If $\cd(M,M)<\infty$ then  $M$ is free. In particular, $\cd(M,M)=\dim(R)$.
\end{corollary}

\begin{proof}
According to Proposition \ref{d} we observe that $\cd(M,M)\leq \dim R$. By the same reasoning $\Tor^R_+(M,M^\ast)=0$. In the light of depth formula over $\AB$-rings (see \cite[Corollary 5.3(b)]{depth}) we know $\depth(M\otimes_RM^\ast)=\dim R>1$. In particular, $H^1_{\fm}(M\otimes_RM^\ast)=0$. Due to \cite[8.1]{ACS} we conclude that $M$ is free, as desired. \footnote{One may drive this from some known cases of Auslander–Reiten conjecture by showing some $\Ext^+_R(M,M)$ vanishes. As, we are interested in the previous  argument, we left details of this to the reader.}
\end{proof}

From now on, $R$ denotes a local ring of prime characteristic $p>0$, and $\varphi:R\to R$ denotes the \emph{Frobenius endomorphism} given by $\varphi(a)=a^{p}$ for $a\in R$. Each iteration $\varphi_n$ of $\varphi$ defines a new $R$-module structure on the set $R$, and this $R$-module is denoted by $\up{\varphi_n}R$, where $a\cdot b = a^{p^{n}}b$ for $a, b \in R$. 
Recall that $R$ is said to be \emph{F-finite} if the Frobenius endomorphism makes $R$ into a module-finite $R$-algebra, i.e., if $\up{\varphi}R$ is a finitely generated $R$-module.

\begin{corollary}\label{2.6}
	Let $(R,\fm)$ be a  normal and locally $\AB$-ring  of prime characteristic. If $R$ is $F$-finite, then $\cd(\up{\varphi_n}R,\up{\varphi_n}R)<\infty$ iff $R$  is regular.  \end{corollary}

Recall that $R$ is locally $\AB$ means $R_{\fp}$ is $\AB$ for any prime ideal $\fp$. This is an open problem \cite{ABu} that $\AB$ rings are locally $\AB$.
\begin{proof}First, suppose that 
	$\cd(\up{\varphi_n}R,\up{\varphi_n}R)<\infty$.	The proof is proceeds via induction on $d:=\dim R$.  Recall that normality means $(S_2)+(R_1)$, so the desired claim is clear for $d<2$. We now want to verify the inductive hypothesis, and drive an additional property. Since $\up{\varphi_n}R$ is maximal Cohen-Macaulay, $H^{d+i}_{\fm}(\up{\varphi_n}R,\up{\varphi_n}R)^{ \vee}\cong \Tor^R_i((\up{\varphi_n}R),(\up{\varphi_n}R)^\ast)$, and the later module behaves nicely with respect to localization, namely,   $(\up{\varphi_n}R)_{\fp}= \up{\varphi_n}(R_{\fp})$ and so  $\Tor^{R_{\fp}}_i((\up{\varphi_n}R_{\fp}),(\up{\varphi_n}R_{\fp})^\ast)=0$. Again, thanks to $H^{\Ht(\fp)+i}_{\fp R_{\fp}}((\up{\varphi_n}R_{\fp}),(\up{\varphi_n}R_{\fp}))^{ \vee}\cong \Tor^{R_{\fp}}_i((\up{\varphi_n}R_{\fp}),(\up{\varphi_n}R_{\fp})^\ast)$, we deduced that $\cd((\up{\varphi_n}R_{\fp}),(\up{\varphi_n}R_{\fp}))<\infty$. It is easy to see that $R_{\fp}$ is $F$-finite, and from the assumption it is locally $\AB$. In view of the inductive hypothesis, $R$ is regular over punctured spectrum.
By a result of Kunz \cite{Ku} $\up{\varphi_n}R$ is locally free over the punctured spectrum, also it is maximal Cohen-Macaulay. The claim now follows from Corollary \ref{a}.

The reverse part is an immediate application of Kunz theorem along with Grothendieck's vanishing theorem.
\end{proof}

\begin{corollary}
	Let $(R,\fm)$ be an  $\AB$-ring of isolated singularity of prime characteristic  and dimension bigger than one. If $R$ is $F$-finite, then $\cd(\up{\varphi_n}R,\up{\varphi_n}R)<\infty$ iff $R$  is regular.
\end{corollary}

\begin{corollary}
	Let $R$ be complete-intersection and $F$-finite. Then $$\cd(\up{\varphi_n}R,\up{\varphi_n}R)<\infty\Longleftrightarrow\emph{R  is regular}.$$  In particular, if the ring  is not regular, then $H^j_{\fm}(\up{\varphi_n}R,\up{\varphi_n}R)\neq 0\Longleftrightarrow j\geq \dim R$.
\end{corollary}
\begin{proof}First, suppose $\cd(\up{\varphi_n}R,\up{\varphi_n}R)<\infty$. The proof is by induction on $d:=\dim R$.
Due to \cite{HM} we may assume $d>1$.
In view of the inductive hypothesis, we may  and do assume that the ring is of isolated singularity.
Now by an argument similar to Corollary \ref{2.6} the ring is regular. The reverse implication is also trivial.
Following the previous result, it is enough to prov  the particular case.  Since (see \cite[Proposition 5.5]{bi}) $$\inf\{J:H^j_{\fm}(M,N)\neq0\}=\grade(\Ann_R(\frac{\up{\varphi_n}R}{\fm\up{\varphi_n}R}),\up{\varphi_n}R)=\depth(\up{\varphi_n}R)=d,$$ we know $H^j_{\fm}(\up{\varphi_n}R,\up{\varphi_n}R)=0$ for all $j< d$. Suppose the ring is not regular and   on the contrary suppose that  $H^j_{\fm}(\up{\varphi_n}R,\up{\varphi_n}R)=0$ for some $j\geq d$.  Thanks to the proof of Proposition \ref{d} we know $\Tor^R_{j-d}(\up{\varphi_n}R,(\up{\varphi_n}R)^\ast)=0$. By complete-intersection assumption (see \cite{Mi}), $\Tor^R_{>j-d}(\up{\varphi_n}R,(\up{\varphi_n}R)^\ast)=0$. This in turn implies that $H^{>j}_{\fm}(\up{\varphi_n}R,\up{\varphi_n}R)=0$.
From the first part,  we conclude that $R$  is regular. This is a contradiction that we searched.
\end{proof}

Now, we extend the previous corollaries:
\begin{proposition}\label{idg}
	Let $R$ be a Cohen-Macaulay $F$-finite ring, $M$ be finitely generated and $n\gg 0$. The following are equivalent:
		\begin{itemize}
		\item[$(1) $] $\cd(\up{\varphi_n}R,M)\leq \dim R$
		\item[$(2) $] $\cd(\up{\varphi_n}R,M)<\infty$
		
		\item[$(3) $]  $\id_R(M)<\infty.$
		
	\end{itemize}
In particular, for some and so all $m>0$, we have $\cd(\up{\varphi_n}R,\up{\varphi_m}R)<\infty  \Longleftrightarrow\emph{R  is regular}$.
\end{proposition}
\begin{proof}	$(1)\Rightarrow (2)$: Trivial, as the ring is local.
	
	$(2)\Rightarrow (3)$:	Suppose  $\cd(\up{\varphi_n}R,M)<\infty$.
	Without loss of generality we may assume the ring is complete, so it has the canonical module.  We look at Auslander-Buchweitz approximation theory (see \cite[3.3.28]{BH}). This 
	gives a short exact sequence $$0\lo I\lo X\lo M\lo 0\quad(\ast)$$ where $X$ is maximal Cohen-Macaulay and $I$ is of finite injective dimension. Let $i$ be bigger than $\max\{\cd(\up{\varphi_n}R,M),\id(I)\}$ and be finite.
 This gives the exact sequence $$0=H^{i}_{\fm}(\up{\varphi_n}R,I)\lo H^i_{\fm}(\up{\varphi_n}R,X)\lo H^i_{\fm}(\up{\varphi_n}R,M)=0\quad(+)$$From $(+)$ we see $H^i_{\fm}(\up{\varphi_n}R,X)=0$ and so
 $\cd(\up{\varphi_n}R,X)<\infty$. Set $(-)^\dagger:=\Hom(-,\omega_R)$. Recall from \cite[Theorem 3.3.10]{BH} that 
  $X^\dagger=\Hom(X,\omega_R) $ is maximal Cohen-Macaulay. Let $d:=\dim(R)$. 
According to   the proof of Proposition \ref{c} that  $H^{d+i}_{\fm}(\up{\varphi_n}R,X)^{ \vee}\cong \Tor^R_i(\up{\varphi_n}R,X^\dagger), $ and so for all $i\gg 0$ we have $\Tor^R_i(\up{\varphi_n}R,X^\dagger)=0.$
By the work of Koh-Lee \cite[Theorem 2.2.8]{Mi}, $\pd(X^\dagger)<\infty$. Thanks to the  Auslander-Buchsbaum formula, $X^\dagger=\oplus_I R$ is free. In the light of \cite[Theorem 3.3.10]{BH} we see 	$$X\cong X^{\dagger\dagger }=\oplus_I \omega_R$$ which is of finite injective dimension. Apply this along with $(\ast)$ we observe that
$M$ is  of finite injective dimension. 

	$(3)\Rightarrow (1)$: If $\id(M)<\infty$, then 
	$\id(M)=\depth(R)$. As the ring is Cohen-Macaulay\footnote{Following Bass' conjecture we do not assume this.}, we see $\Ext^i_R(-,M)=0$ for all $i>\dim R$, and so 
$\cd(\up{\varphi_n}R,M)\leq \dim R$.

	To see the particular case it is enough to apply Avramov-Iyengar-Miller  \cite[Theorem 13.3]{Av}.
\end{proof}
The dual version is easy:
\begin{observation}\label{d2}
i)	Let $R$ be a Cohen-Macaulay $F$-finite ring with canonical
	 module, $M$ be finitely generated and $n\gg 0$.  Then $\cd(M,(\up{\varphi_n}R)^\dagger)<\infty$
	iff  $\pd_R(M)<\infty.$
	
	ii) If in addition $R$ is Gorenstein, then $\cd(M,\up{\varphi_n}R)<\infty$
	iff  $\pd_R(M)<\infty.$
\end{observation}

\begin{proof}
	i) This is included in the proof of Proposition \ref{idg}, where we left the straightforward modification to the reader.

ii) Recall that $(\up{\varphi_n}R)^\dagger=(\up{\varphi_n}R)^\ast$
and that over Gorenstein rings
$(\up{\varphi_n}R)$ is self-dual, see \cite[1.1]{gotoo}. Now, it is enough to apply item i).
 \end{proof}

As a sample, we present a method to drop
F-finiteness:

\begin{proposition}\label{df}
i)	Let $R$ be a Cohen-Macaulay ring of prime characteristic with canonical module, $M$ be finitely generated
 and $n\gg 0$.  Then $\cd(M,\omega_{\up{\varphi_n }R})<\infty$
iff  $\pd_R(M)<\infty.$

ii) If in addition $R$ is Gorenstein, then $\cd(M,\up{\varphi_n}R)<\infty$
iff  $\pd_R(M)<\infty.$
\end{proposition}

\begin{proof}
We prove both claims at the same times.	
Suppose $\cd(M,\omega_{\up{\varphi_n }R})<\infty$. Recall that $M$ is finitely generated. It turns out there is a Grothendieck's spectral sequence
$\Ext_R^p(M,H^{q}_{\fm}(\omega_{\up{\varphi_n }R}))\Rightarrow H^{p+q}_{\fm}(M,\omega_{\up{\varphi_n }R}).$ 
Recall that  $\omega_{\up{\varphi_n }R}$ is maximal  Cohen-Macaulay over the algebra $\up{\varphi_n }R$. In particular, it is 
big Cohen-Macaulay as an $R$-module. Following definition, 
$H^{q}_{\fm}(\omega_{\up{\varphi_n }R})=0$ for all $q\neq d:=\dim R$, i.e., the spectral sequence
collapses and we have $$H^{d+q}_{\fm}(M,\omega_{\up{\varphi_n }R})\cong\Ext_R^p\big(M,H^{d}_{\fm}(\omega_{\up{\varphi_n }R})\big).$$

There is a local ring extension $(S,\fn)$ of $(R,\fm)$ such that $\fm S=\fn$, $S$ is $F$-finite, $S$ is faithfully flat over $R$, and $S$ has infinite residue field.
For example, letting $\widehat{R}\cong k[[x_1,\cdots, x_m]]/I$ for some ideal $I$, we can pick $S=\bar{k}[[x_1,\cdots,x_m]]/I\bar{k}[[x_1,\cdots, x_m]]$, where $\bar{k}$ denotes the algebraic closure of $k$.
Note, if $L$ is an $R$-module, then it follows that $$(L\otimes_R\up{\varphi_n }R)\otimes_R S\cong (L\otimes_RS)\otimes_S\up{\varphi_n }S\quad(\flat)$$
By base change theorem, the map $$\up{\varphi_n}R\stackrel{\cong}\lo R\otimes_R\up{\varphi_n}R\lo  S\otimes_R\up{\varphi_n}R\stackrel{\cong}\lo \up{\varphi_n }S\quad(+
)$$ is flat. Also, recall that $\Hom_A(B,\omega_A)\stackrel{(\ast)}=\omega_B$, where $A\to B$ is a morphism of local rings of same dimension and equipped with canonical modules so that $B$ is finitely generated as an $A$-module. This gives:

$$  \omega_{\up{\varphi_n }R}\otimes_RS  \cong   \omega_{\up{\varphi_n }R}\otimes_{\up{\varphi_n }R}(\up{\varphi_n }R\otimes_RS)  \stackrel{(\flat)}\cong   \omega_{\up{\varphi_n }R}\otimes_{\up{\varphi_n }R}\up{\varphi_n }S  \stackrel{(+)}\cong  \omega_{\up{\varphi_n }S} \stackrel{(\ast)} \cong  \Hom_S(\up{\varphi_n }S,\omega_S)\quad(\natural)
$$

Due to the assumptions,  we know $\Ext_R^p\big(M,H^{d}_{\fm}(\omega_{\up{\varphi_n }R})\big)=0$ for all $p\gg 0$. Let $(-)^\dagger:=(-)^{\dagger_S}:=\Hom_S(-,\omega_S)$, and
recall that $M$ is finitely generated. 	This in turn implies the second implication of the following:

\begin{eqnarray*}
	\Ext_R^i\big(M,H^{d}_{\fm}( \omega_{\up{\varphi_n }R})\big)=0 & \Longleftrightarrow & \Ext_R^i\big(M,H^{d}_{\fm}( \omega_{\up{\varphi_n }R})\big)\otimes_RS=0
	\\
	& \Longleftrightarrow &  \Ext_S^i\big(M\otimes_RS,H^{d}_{\fm}( \omega_{\up{\varphi_n }R})\otimes_RS\big)=0\\
	& \Longleftrightarrow &  \Ext_S^i\big(M\otimes_RS,H^{d}_{\fm}( \omega_{\up{\varphi_n }R}\otimes_RS)\big)=0\\
	& \stackrel{(\natural)}\Longleftrightarrow & \Ext_S^i\big(M\otimes_RS,H^{d}_{\fm}((\up{\varphi_n}S)^{\dagger_S})\big)=0\\
	& \Longleftrightarrow & \Ext_S^i\big(M\otimes_RS,H^{d}_{\fn}((\up{\varphi_n}S)^{\dagger_S})\big)=0,
\end{eqnarray*}
where in the   third (resp. five) implication  we used the  flat base change (resp.    independence) theorem.
Also, there is an spectral sequence
$$\Ext_S^p(M\otimes_R S,H^{q}_{\fn}(\up{\varphi_n}S)^\dagger)\Rightarrow H^{p+q}_{\fn}(M\otimes_R S,(\up{\varphi_n}S)^\dagger).$$ 
Recall that   $(\up{\varphi_n}S)^\dagger$ is 
Cohen-Macaulay. This shows $$H^{d+q}_{\fn}(M\otimes_R S,(\up{\varphi_n}S)^\dagger)\cong\Ext_R^p\big(M\otimes_R S,H^{d}_{\fn}(\up{\varphi_n}S)^\dagger)\big).$$
So, we conclude that  $\cd(M\otimes_RS,(\up{\varphi_n}S)^{\dagger_S})<\infty$.
One may deduce from $\pd_S(M\otimes_RS)<\infty$ 
that $\pd(M)<\infty$. Also, it follows that ascent through the flat map $R\to S$ implies that $S$ is Cohen-Macaulay (Gorenstein). 
 From these, and by replacing $R$ with $S$, we may and do assume that $R$ is $F$-finite. Now, the desired claim is in Observation \ref{d2}.
\end{proof}

Concerning the next item,
it slightly extends \cite[2.9+2.11]{l},  as we do not assume $R$ is generically Gorenstein:

\begin{observation}
	Suppose $R$ is Cohen-Macaulay, $\underline{x}$ is a system of parameter and $M$ is maximal Cohen-Macaulay. If $\Tor_i^R(M/\underline{x}M,\up{\varphi_n}R)=0$ for some $n\gg0$ then $M$ is free.
	In particular, if $R$ has  canonical module  and $\Tor_i^R(\omega/\underline{x}\omega,\up{\varphi_n}R)=0$ then $R$ is Gorenstein.
\end{observation}

\begin{proof}We may assume that $R$ is of positive dimension.
	Since $\up{\varphi_n}R$ is big Cohen-Macaulay, $\Tor_+^R(R/\underline{x}R,\up{\varphi_n}R)=0$. Then there is a first quadrant spectral
	sequence $$\Tor_i^R(\Tor_j^R(M,R/\underline{x}R),\up{\varphi_n}R)\Rightarrow \Tor_{i+j}^R(M,\up{\varphi_n}R/\underline{x}\up{\varphi_n}R).$$
	Since $M$ is maximal Cohen-Macaulay, $\Tor_j^R(M,R/\underline{x}R)=0$. This means the spectral sequence collapses. From this, $\Tor_{i}^R(M,\up{\varphi_n}R/\underline{x}\up{\varphi_n}R)=\Tor_i^R(M/\underline{x}M,\up{\varphi_n}R)=0$. But we know from \cite[
	Corollary 3.3]{tay} that $k$ is a direct summand of $\up{\varphi_n}R/\underline{x}\up{\varphi_n}R$. This means $\pd(M)<\infty$, and so it is free.
\end{proof}

\begin{corollary}
	Suppose $R$ is a  Cohen-Macaulay   F-finite ring  with canonical module, $\underline{x}$ is a system of parameter and    $M$ is maximal Cohen-Macaulay. For $n\gg 0$, the following are equivalent:
	\begin{itemize}
		\item[$(1) $] $H^i_{\fm}(M/\underline{x}M,(\up{\varphi_n}R)^\dagger) = 0$ for all $i> \dim R$ 
		\item[$(2) $] $H^i_{\fm}(M/\underline{x}M,(\up{\varphi_n}R)^\dagger) = 0$ for some $i> \dim R$	\item[$(3) $] $M$ is free.
		
	\end{itemize}
\end{corollary}

\begin{proof}$(1)\Rightarrow (2)$: Trivial.
	
$(2)\Rightarrow (3)$:	Set $j:=\dim R- i$.
Recall from the proof of Proposition \ref{c} that
$$\Tor_j^R(M/\underline{x}M,\up{\varphi_n}R)=\Tor_j^R(M/\underline{x}M,(\up{\varphi_n}R)^{\dagger\dagger})=H^i_{\fm}(M/\underline{x}M,(\up{\varphi_n}R)^\dagger)^{\vee} = 0.$$ 
Now, we apply the previous observation to deduce $M$ is free.

$(3)\Rightarrow (1)$: Since $M$ is free, thanks to Koszul complex with respect to $\underline{x}$ over $M$, we know $\pd(M/\underline{x}M)<\infty$. This implies that $\cd(M/\underline{x}M,(\up{\varphi_n}R)^\dagger)<\infty$, and so $(1)$ easily follows,  for instance see   \cite[Theorem 3.2]{HZ}.
\end{proof}

\section{Vanishing of $\Ext^\ast(\up{\varphi_n}R,-)$}
We say $M$ is (strongly) tor-rigid, if $\Tor_i^R(M,N)=0$ then 
$\Tor_{>i}^R(M,N)=0$ when $N$ is (not necessarily) finitely generated. Also, $M$ is called e-rigid if $\Ext_R^1(M,M)=0$. Recall that if $\depth(R)=1$ then $\up{\varphi_n}R$  is tor-rigid (see \cite{Mi}), so the following may be considered as higher analogue of \cite[1.4]{two}.
\begin{proposition}
	Let $R$ be $(S_2)$, F-finite and reduced. If
	$\up{\varphi_n}R$ is e-rigid and strongly tor-rigid, then $R$ is regular.
\end{proposition}

Instead of the e-rigidity property of $\up{\varphi_n}R$, and throughout this section, one may assume only that $\Ext^1_R(\up{\varphi_n}R,\up{\varphi_m}R)=0$ for some $m$. 
\begin{proof} Let $d:=\dim R$. There is nothing to prove if  $d=0$, as the ring is assumed to be reduced.
	If $\dim R= 1$ the claim is in \cite[1.4]{two} that we do not need the tor-rigidity assumption. Recall that $(\up{\varphi_n}R)_{\fp}=\up{\varphi}(R_{\fp})$ and that
	$\up{\varphi}(R_{\fp})$ is e-rigid. Apply this for any height one prime ideal, we deduce that $R$ satisfies $(R_1)$. Following Serre's characterization of normality, we observe that $R$ is normal. By $\Tr(-)$ we mean the Auslander's transpose. Set $T:=\Tr(\Syz_1(\up{\varphi_n}R))$, and look at the following exact sequence taken from \cite{ABu} \begin{equation} 
	\Tor_2^R(T,\up{\varphi_n}R)\rightarrow\Ext^1_R(\up{\varphi_n}R,R)\otimes_R \up{\varphi_n}R\rightarrow\Ext^1_R(\up{\varphi_n}R,\up{\varphi_n}R)\rightarrow\Tor_1^R(T,\up{\varphi_n}R)\rightarrow 0.
	\end{equation} Recall that the ext-rigidity means
	$\Ext^1_R(\up{\varphi_n}R,\up{\varphi_n}R)=0$. So, $\Tor_1^R(T,\up{\varphi_n}R) =0$. Recall that the tor-rigidity implies that $\Tor_2^R(T,\up{\varphi_n}R) =0$. Due to the displayed exact sequence $\Ext^1_R(\up{\varphi_n}R,R)\otimes_R \up{\varphi_n}R=0$.
	By Nakayama's lemma, $\Ext^1_R(\up{\varphi_n}R,R)=0$.
In the light of   Matlis duality, we observe
	$$0=\Ext^1_R(\up{\varphi_n}R,R)^{ \vee}=\Ext^1_R(\up{\varphi_n}R,E(k)^{ \vee})^{ \vee}=\Tor^R_1(\up{\varphi_n}R,E(k)).$$ Thanks to tor-rigidity, one may deduce that $\Tor^R_+(\up{\varphi_n}R,E(k))=0$. By Matlis duality, $\Ext^+_R(\up{\varphi_n}R,R)=0$. Similarly, $\Ext^+_R(\up{\varphi_n}R,\up{\varphi_n}R)=0$.
	According to Auslander–Reiten over normal rings, see \cite[3.13(3)]{ta}, we conclude that $\up{\varphi_n}R$ is free. Recall from  Kunz \cite{Ku}  that $R$ is regular if and only if $\up{\varphi^r}R$ is a flat $R$-module. The proof is now completed.
\end{proof}
Strongly tor-rigid modules with dimension $<\dim R-1$ are so special:
\begin{remark}
	Suppose $R$ is Cohen-Macaulay and let $M$ be a finitely generated module such that $\dim M<\dim R-1$. Then $M$ is not strongly tor-rigid.
\end{remark}

\begin{proof}Let $d:=\dim R$, $d_0:=d-\dim M$ and 
 $d_1:=d-\dim M-1$.
	The  Check complex provides  a flat resolution of $H:=H^d_{\fm}(R)$. So, $\Tor^R_{d_1}(M,H )=H^{\dim M+1}_{\fm}(M)=0$,
but	 $\Tor^R_{d_0}(M,H )=H^{\dim M}_{\fm}(M)\neq 0$ and $d_0>d_1$.
\end{proof}
\begin{corollary}\label{dz}
Suppose $\depth(R)=0$, it is $F$-finite and
	$\up{\varphi_n}R$ is e-rigid for $n\gg 0$. Then $R$ is a field.
\end{corollary}

\begin{proof}\iffalse'There is a local ring extension $(S,\fn)$ of $(R,\fm)$ such that $\fm S=\fn$, $S$ is $F$-finite, $S$ is faithfully flat over $R$, and $S$ has infinite residue field. Note, for  an $R$-module $M$, that $(M\otimes_R\up{\varphi^n}R)\otimes_R S\cong (M\otimes_RS)\otimes_S\up{\varphi^n}S$. 
This implies that $0=\Ext^1_R(\up{\varphi^n}R,\up{\varphi^n}R)\otimes S\cong \Ext^1_S(\up{\varphi^n}S,\up{\varphi^n}S)$. It is easy to see $\depth(S)=0$. Also, regularity of $R$ follows from the regularity of $S$.	
From this, and by replacing $R$ with $S$, we may and do assume that $R$ is $F$-finite.\fi
It follows from the previous argument that $\Tor_1^R(\Tr(\Syz_1(\up{\varphi_n}R)),\up{\varphi_n}R) =0$. By the work of Koh-Lee \cite[Theorem 2.2.8]{Mi},
$\pd(\Tr(\Syz_1(\up{\varphi_n}R)))<\infty$. As  $(\Syz_1(\up{\varphi_n}R))^\ast$ is the second syzygy of  $\Tr(\Syz_1(\up{\varphi_n}R))$ we deduce that
$\pd((\Syz_1(\up{\varphi_n}R))^\ast)<\infty$. Thanks to Auslander-Buchsbaum formula, we know $\Syz_1(\up{\varphi_n}R)^\ast$ is free. Again, we are going to use the depth-zero assumption. Namely, following Ramras \cite[2.6]{ram}, we know $\Syz_1(\up{\varphi_n}R)$ is free. So, $\pd(\up{\varphi_n}R) <\infty$. In the light of Kunz \cite{Ku} we observe that $R$ is regular, and so it is field.
\end{proof}
The ring  $R$  is called   $(G_1)$, if $R$  is
Gorenstein in codimension less than or equal to one.  Recall that $R$  quasi-normal means $(S_2)+(G_1)$. Concerning the next result, part $(c)$ is well-known (see \cite{ev}):

\begin{fact} Assume one of the following:
\begin{enumerate}[(a)]
	\item   $M$ is reflexive and $\Ext^1_R(M,R)=0$.
	
	\item $R$ is Gorenstein and $M$ is almost Cohen-Macaulay.
	\item Let $M$ be torsionless and  $R$ be quasi-normal.
\end{enumerate}
Then  $\Syz_1(M)$ is reflexive.
\end{fact}

\begin{proof}
$a)$ There is a free module $F$ and an exact sequence  $0\to M^\ast \to F^\ast\to\Syz_1(M)^\ast \to \Ext^1_R(M,R)=0$. This yields
	$$\begin{CD}
0@>>> \Syz_1(M)^{\ast\ast} @> >> F^{\ast\ast} @>>> M^{\ast\ast}\\
@.@AAA\cong @AAA = @AAA   \\
0@>>> \Syz_1(M) @>>> F @>>>M@>>>0,\\
\end{CD}	$$	
which shows $\Syz_1(M)$ is reflexive.

b): As	$\Syz_1(M)$ is maximal Cohen-Macaulay,
and of finite $G$-dimension, by Auslander-Bridger should be totally reflexive.

c): As $M$ is first syzygy, $\Syz_1(M)$ is second syzygy and so reflexive, see \cite[Theorem 3.6]{ev}.
\end{proof}

\begin{corollary}
	Suppose $R$ is complete intersection, $F$-finite and
	$\up{\varphi_n}R$ is e-rigid for some $n> 0$. Then $R$ is regular.
\end{corollary}

\begin{proof}Let $d:=\dim R$. Thanks to 
Corollary \ref{dz} we may assume $d>0$. The case $d=1$ is subject of	\cite[1.4]{two}.
So, we may assume $d>1$. This is in the above argument that  $\pd((\Syz_1(\up{\varphi_n}R))^\ast)<\infty$. As $\up{\varphi_n}R$ is maximal Cohen-Macaulay, its syzygy, and the dual of its syzygy is maximal Cohen-Macaulay. By Auslander-Buchsbaum,
$(\Syz_1(\up{\varphi_n}R))^\ast$ is free.	So, 
$(\Syz_1(\up{\varphi_n}R))^{\ast\ast}$ is free. But, recall that the claim is clear in dimension one, so:
\begin{enumerate}[(a)]
	\item $(\up{\varphi_n}R)_{\fp}$ is reflexive for all $\fp$ with  $\depth R_{\fp} \leq 1$, and
	\item $\depth (\up{\varphi_n}R)_{\fp} \geq 2$ for all $\fp$ with  $\depth R_{\fp} \geq 2$.
	\end{enumerate}

From this, or even without this, $\up{\varphi_n}R$ is reflexive. It turns out that $(\Syz_1(\up{\varphi_n}R))$ is reflexive. Consequently, $\Syz_1(\up{\varphi_n}R)$ is free.
Therefore, $\pd(\up{\varphi_n}R) <\infty$. By Kunz \cite{Ku}, $R$ is regular.
\end{proof}

The self-dual property of $\up{\varphi_n}R$ was asked in \cite{hhe} and studied in \cite{gotoo}. One may ask the reflexivity problem of $\up{\varphi_n}R$. In the positive side:
\begin{corollary}
	Suppose $R$ is quasi-normal and F-finite. Then  $\up{\varphi_n}R$ is reflexive.
\end{corollary}

\begin{example}
	There are several situations, e.g. when betti numbers are strictly increasing,  for which $\up{\varphi_n}R$ is not reflexive.
\end{example}

\begin{proof}
We are going to use \cite{ram}, where Ramras introduced  a class of rings coined with $BNSI$.
Over this class of rings any reflexive
	module is free. Now recall this class included in rings of depth zero. Thanks to Kunz \cite{Ku}, it is enough to take any $BNSI$ ring of prime characteristic, which is not a field. To be more explicit, just let $R:=\bar{\mathbb{F}_p}[[X,Y]]/(X,Y)^2$, and follow some lines from \cite{moh}.
	\end{proof}

Directed limit of Matlis-reflexive modules is every thing, but not inverse limit:
\begin{lemma}\label{t3}Suppose that $(R,\fm)$ is local and complete. Let $\{E_i\}$ be an inverse system of Matlis-reflexive modules. The following assertions are valid:
\begin{itemize}
	\item[$(i)$]  $\vil (E_i^{ \vee}) \cong [\vil (E_i^{ \vee})]^{\vee\vee}$  iff $\vil (E_i^{ \vee})\cong (\vpl E_i)^{ \vee}$.

	\item[$(ii)$] One has  $\vpl  E_i\cong(\vpl  E_i)^{\vee\vee}$. \end{itemize}
\end{lemma}
\begin{proof} i) First, suppose $\vil (E_i^{ \vee})$ is Matlis-reflexive. Let $A:=\vil (E_i^{ \vee})$. 
Recall that: $$A=A^{\vee\vee}=(\vil (E_i^{ \vee}))^{\vee\vee}=[\Hom(\vil  E_i^{ \vee},E)]^{ \vee}=[\vpl\Hom(E_i^{ \vee},E)]^{ \vee}=[\vpl E_i]^{ \vee}.$$
The reverse implication also is easy.

ii)
Set $A:=\vpl  E_i$, let $M_i:=E_i^{ \vee}$, and $B:=\vil  M_i$.  Then
$E_i=M_i^{ \vee}$ and 
$$A=\vpl E_i=\vpl\Hom(M_i,E) =\Hom(\vil M_i,E)=B ^{ \vee},$$
and so
$A^{\vee\vee}=B^{\vee\vee\vee}=
[\Hom(\Hom(B,E),E)]^\vee=[B\otimes\Hom(E,E)]^\vee
 =(B\otimes \hat{R})^\vee= B ^{ \vee}=A.$
\end{proof}

\begin{lemma}\label{t2}
	Let   $R$ be a Cohen-Macaulay ring and of prime characteristic with canonical module. Let
	$N$ be maximal Cohen-Macaulay and locally free over the punctured spectrum. Then
$\Tor^R_i(\mathcal{M},N)\cong\Ext^{d+i}_R(\mathcal{M},N^\dagger)^{ \vee}$ for all Matlis-reflexive module $\mathcal{M}$.
\end{lemma}

\begin{proof}
Step I) First, assume $\mathcal{M}:=M$ is finitely generated.	A slight modification of the proof of \cite[Lemma 3.5(2)]{ta} yields the desired claim.

 Step II) Write
$\mathcal{M}=\vil_{j\in J} M_j$ where $M_j$ is finitely generated. Recall that $\Tor^R_i(\mathcal{M},N)$
is Matlis-reflexive.  We use this (resp. step I) along with \ref{t3} in the third (resp. second)  identification of: 
\begin{eqnarray*}
\Tor^R_i(\mathcal{M},N)&\cong&\vil_{} \Tor^R_i(M_j,N)\\&  \cong & \vil[\Ext^{d+i}_R(M_j,N^\dagger)^{ \vee}]
	\\
	& \cong &  [\vpl\Ext^{d+i}_R(M_j,N^\dagger)]^{ \vee}\\
	& \cong & \Ext^{d+i}_R(\vil M_j,N^\dagger)^{ \vee}\\ &=& \Ext^{d+i}_R(\mathcal{M},N^\dagger)^{ \vee},
\end{eqnarray*}yields the desired claim.
\end{proof}
\begin{question}\label{q2}
  Let $R$ be  Cohen-Macaulay  and of prime characteristic with canonical module and $n\gg 0$.\begin{itemize}
		\item[$(1)$] (See \cite[Page 167]{mar}) Suppose $\Ext^{j}_R(\up{\varphi_n}R,R)=0$ for some $j>\dim R$.  Is $R$ Gorenstein? 
		\item[$(2) $]  (See \cite[2.12]{l}) Suppose $\Tor^R_i(\up{\varphi_n}R,\omega_R)=0$ for some $i>0$. Is $R$ Gorenstein?
		
	\end{itemize}
\end{question}

When is $(\up{\varphi_n}R)^{\vee\vee}=\up{\varphi_n}R$? For instantiate,
$R$ is F-finite.
\begin{proposition}Suppose %$\up{\varphi_n}R^{\vee\vee}=\up{\varphi_n}R$, for instantiate,
	 $R$ is F-finite. Then 
Question \ref{q2}(1) $\equiv$ Question \ref{q2}(2).
\end{proposition}

\begin{proof}
	The proof is by induction on $d:=\dim R$.
First, assume $d\leq 1$. We claim both questions are true. Let us deal with  Question \ref{q2}(1).
Indeed, as depth of the ring is at most one,  $\up{\varphi_n}R$ is tor-rigid.
	From this, $\Ext^{>d}_R(\up{\varphi_n}R,R)=0$.
	It follows  from \cite[Theorem 1.2]{mar} that $R$  is Gorenstein.  We should show that  Question \ref{q2}(2) is true in this case, as well.
	Indeed, as  $\up{\varphi_n}R$ is tor-rigid, $\Tor^R_{>i}(\up{\varphi_n}R,\omega_R)=0$.
	It follows  from Koh-Lee that $\pd(\omega_R)<\infty$. Following Auslander-Buchsbaum formula, $\omega_R$  is free, and consequently, the ring is Gorenstein.
	Suppose the claim is true in rings of dimension less that $d$.  
	We now want to verify the inductive hypothesis, and drive an additional property. The assumptions behave  well with respect to the localization. In view of the inductive hypothesis, $R$ is Gorenstein over punctured spectrum. So, $\omega_R$ is locally free
	over the punctured spectrum. This allows us to apply Lemma \ref{t2} to see $\Tor^R_i(\up{\varphi_n}R,\omega_R)\cong\Ext^{d+i}_R(\up{\varphi_n}R,R)^{ \vee}.$
	The proof is now completed.
\end{proof}

It may be worth to note that the ring in \cite[Page 167]{mar} was not assumed Cohen-Macaulay. Also, the restriction $j>\dim R$ is superfical at least the ring is $F$-finite, see \cite[Corollary 3.5]{mar}.
So, one may reformulate  Question \ref{q2}(1)  as follows: 
\begin{question}\label{q12}
	 Suppose $\Ext^{j}_R(\up{\varphi}R,R)=0$ for some $j>0$.  Is $R$ Gorenstein? 
\end{question}

It seems Question \ref{q12} is not true, and so  negative answer to \cite[Conjecture 7.1]{Mi}.

\begin{fact}(See  \cite[Theorem 3.8]{ev})
	\label{s3} i ) Let $k>2$.
Let  $R$  be a local
ring which satisfies $(S_k)$, and be $(G_1)$.  The module $M$ satisfies $(S_k)$ iff $M$ is reflexive and $\Ext^i_R(M^*,R) = 0$ for $1 \leq i \leq k - 2$.

ii) (See \cite[Kollar 5]{hhe} and \cite{gotoo})
Let  $R$  be UFD and of prime characteristic. Then $\up{\varphi}R$ is self-dual.
\end{fact}

\begin{discussion}\label{ff}
There is a UFD local ring $R$ of prime characteristic with prefect residue field so that it  	is $(S_3)$ but not Cohen-Macaulay.
\end{discussion}
\begin{proof}Let $A$ be any complete local ring of prime characteristic with prefect residue field
so that it $(S_3)$ but not Cohen-Macaulay.
Following the work of Heitmann \cite{heit}
there is a UFD local ring $R$ such that
$\hat{R}=A$. Following \cite[Proposition 2.1.16]{BH} $R$ satisfies $(S_3)$ but it is 
not Cohen-Macaulay. \end{proof}

\begin{proposition}
Suppose in addition to Discussion \ref{ff}, the ring is F-finite. Then Question \ref{q12} is not true.
\end{proposition}
\begin{proof}
Since Frobenius behaves well with respect to localization, $\up{\varphi}R$ satisfies $(S_3)$.     We apply Fact \ref{s3}(i) for $k:=3$ along with Fact \ref{s3}(ii) to observe that
$\Ext^1_R(\up{\varphi}R,R) = 0$ but $R$ is not Gorenstein.\end{proof}

The following reduces checking of d-spots vanishing of $\Ext^{i}_R(\up{\varphi}R,R)$ into at most 3-spots vanishing:
\begin{observation}
Suppose $R$ is Cohen-Macaulay, $F$-finite and of dimension $d$. The following are equivalent:\begin{itemize}
	\item[$(1)$]   $R$ is Gorenstein. 
	\item[$(2) $]  $\up{\varphi}R$ is self-dual, $\Ext^{i}_R(\up{\varphi}R,R) = 0$ for $0<i\in\{1,d-1, d\}$.
\end{itemize}Suppose in addition, $R$ is quasi-normal. Then these are equivalent with:\begin{itemize}
\item[$(3)$]   $\up{\varphi}R$ is self-dual and $\Ext^{i}_R(\up{\varphi}R,R) = 0$ for  $0<i\in\{d-1, d\}$.
\end{itemize}
\end{observation}

\begin{proof}
	The case $d\leq 2$ follows by a result of Goto, see \cite[Theorem 1.1]{gotoo}. This in turn implies that $R$  is
	Gorenstein in codimension less than or equal to two. So, we may assume $d\geq3$.  Since Frobenius behaves well with respect to localization, $\up{\varphi}R$ satisfies $(S_d)$. Thanks to Fact \ref{s3}(i)
	$\Ext^i_R(\up{\varphi}R,R) = 0$ for all $i<d-1$. It remains to apply \cite[Theorem 1.1]{gotoo}.
\end{proof}


\begin{thebibliography}{99}
 
\bibitem{ABuc}
M. Auslander, R. Buchweitz, \emph{The homological theory of maximal Cohen-Macaulay approximations}, Colloque en l’honneur de Pierre Samuel (Orsay, 1987), Mem. Soc. Math. France (N.S) {\bf38} (1989), 5–37.


\bibitem{ABu}
Maurice Auslander and Mark Bridger, \emph{Stable module theory}, Mem. of
the AMS  {\bf94}, Amer. Math. Soc., Providence 1969.

 \bibitem{moh}
M. Asgharzadeh, \emph{Reflexivity revisited}, arXiv:1812.00830   [math.AC].

\bibitem{ACS}M. Asgharzadeh, O. Celikbas,  A.
Sadeghi,  \emph{A study of the cohomological rigidity property}, arXiv:2009.06481 [math.AC].

\bibitem{Av}
 L. L. Avramov, S. Iyengar, C. Miller, \emph{Homology over local homomorphisms}, Amer. Jour. Math.
{\bf128} (2006), 23–90.

\bibitem{bi}
M.H. Bijan-Zadeh, \emph{A common generalization of local cohomology theories}, Glasgow Math. J.,  {\bf21}(2), (1980),
173-181.

\bibitem{BH}{W. Bruns and J. Herzog}, {\it Cohen-Macaulay rings},  Cambridge Studies in Advanced Mathematics,
{\bf 39},
Cambridge University Press, Cambridge, 1993.

\bibitem{two}O. Celikbas, T. Kobayashi, H. Matsui, A. Sadeghi, \emph{
Two theorems on the vanishing of Ext},
 arXiv:2308.08999 [math.AC].

\bibitem{depth}
L. W. Christensen, D. A. Jorgensen, \emph{Vanishing of Tate homology and depth formulas over local rings}, J. Pure Appl.
Algebra  {\bf219} (2015), no. 3, 464–481.


\bibitem{ev} E.G.
Evans,   P. Griffith,   \emph{Syzygies},
LMS. Lecture Note Ser., \textbf{106}
Cambridge University Press, Cambridge, 1985.


\bibitem{mar}B. Falahola,T. Marley, \emph{
	Characterizing Gorenstein rings using contracting endomorphisms}, Journal of Algebra {\bf571} (2021), 168-178.


\bibitem{gotoo}
S. Goto, \emph{A problem on Noetherian local rings of characteristic p}, 
Proc. Amer. Math. Soc. {\bf{64}} (1977), no. 2, 199-205.


\bibitem{sga2}
A. Grothendieck,
\emph{Cohomologie locale des faisceaux coh\'{e}rents et th\'{e}oremes de Lefschetz locaux et globaux} (SGA 2).
S\'{e}minaire de G\'{e}om\'{e}trie Alg\'{e}brique du Bois Marie  1962. Amsterdam: North Holland Pub. Co. (1968).

\bibitem{heit}
R. Heitmann, \emph{Characterizations of completions of unique factorization domains},
Trans. AMS. {\bf{337}} (1993) 379-387.

\bibitem{hhe}J. Herzog, \emph{Ringe der Charakteristik p und Frobenius-Funktoren}, Math. Z.  {\bf140} (1974), 67–78.

\bibitem{HZ}J. Herzog and N. Zamani, \emph{Duality and vanishing of generalized local cohomology}, Arch. Math. (Basel),
{\bf81}(5), (2003), 512-519.

\bibitem{AB}
C. Huneke, D. A. Jorgensen, \emph{Symmetry in the vanishing of Ext over Gorenstein rings}, Math. Scand.  {\bf93} (2003), no. 2,
161–184.

\bibitem{HM}R. Holanda, CB Miranda-Neto,\emph{
An Ext-Tor duality theorem, cohomological dimension, and applications},
 arXiv:2312.09725 [math.AC].


 \bibitem{ta} Kaito Kimura, Yuya Otake and Ryo Takahashi, \emph{Maximal Cohen-Macaulay tensor products and vanishing of Ext modules},
Bull. Lond. Math. Soc. 54 (2022), no. 6, 2456-2468.

\bibitem{Ku}
Ernst Kunz, \emph{Characterization of regular local rings of characteristic p}. Am. J. Math. {\bf91}, 772--784 (1969).


\bibitem{Koh}
J. Koh, K. Lee, \emph{Some restrictions on the maps in minimal resolutions}, J. Algebra {\bf202} (1998),
671–689.

\bibitem{l}J.
Li, {\it Frobenius criteria of freeness and Gorensteinness}, Arch. Math. (Basel) {\bf 98} (2012), no. 6, 499-506.

\bibitem{Mi}
Claudia Miller, \emph{The Frobenius endomorphism and homological dimensions}, Commutative Algebra (Grenoble/Lyon, 2001), Contemp. Math., vol.  {\bf331}, pages 207-234, Amer. Math. Soc., Providence, RI, 2003.

\bibitem{PS}
Christian Peskine and Lucien Szpiro, \emph{Dimension projective finie et cohomologie locale},
Publ. Math. IHES.  {\bf42} (1973),  47--119.

\bibitem{ram}
M. Ramras, \emph{Betti numbers and reflexive modules}, Ring theory (Proc. Conf., Park City, Utah, 1971),   297-308. Academic Press, New York, 1972.



\bibitem{tay}
R. Takahashi and Y. Yoshino, Characterizing Cohen-Macaulay local rings by Frobenius
maps, Proc. Amer. Math. Volume 132, Number 11, Pages 3177–3187,  2004.

 \end{thebibliography}
\end{document}